\documentclass[12pt,twoside]{amsart}
\usepackage{amsmath,amssymb}    
\usepackage{url}                
\usepackage{graphicx}           
\usepackage{amsthm}
\usepackage{fullpage}
\theoremstyle{definition}
\numberwithin{equation}{section}
\newtheorem{theorem}[equation]{Theorem}
\newtheorem{lemma}[equation]{Lemma}
\newtheorem{proposition}[equation]{Proposition}

\date{}

\begin{document}

\author{Pavel Etingof}

\title{On some properties of quantum doubles of finite groups}

\maketitle

\section{Introduction}

In this paper we prove two results about quantum doubles of
finite groups over the complex field. The first result is the
integrality theorem for higher Frobenius-Schur indicators 
for wreath product groups $S_N\ltimes A^N$, where $A$ is a finite
abelian group. A proof of this result for $A=1$ appears in
\cite{IMM}. The second result is a lower bound for the largest
possible number of irreducible representations of the 
quantum double of a finite group with $\le n$ conjugacy classes. 
This answers a question asked to me by Eric Rowell. 
 
{\bf Acknowledgements.} I thank Mio Iovanov, Susan Montgomery,
and Eric Rowell for useful discussions. This research was
partially supported by the NSF grant DMS-1000113. 

\section{Integrality of higher Frobenius-Schur indicators 
for the quantum doubles of wreath product groups $S_N\ltimes A^N$}

\subsection{The result}

Let $H$ be a semisimple Hopf algebra over $\Bbb C$, and $\Lambda$ a left 
integral of $H$ such that $\varepsilon(\Lambda)=1$. 
Let $V$ be an irreducible $H$-module. Then the higher Frobenius-Schur 
indicators $\nu_n(V)$ are defined by the formula (\cite{KSZ}):
$$
\nu_n(V)={\rm Tr}_V(\Lambda_1...\Lambda_n),
$$
where $\Lambda_1\otimes...\otimes\Lambda_n=\Delta_n(\Lambda)$
(using Sweedler's notation with implied summation). 
These are algebraic integers lying in $\Bbb Q(e^{2\pi i/n})$, 
which are known to be integers if $H$ is a 
group algebra. 

An interesting question is when $\nu_n(V)$ are integers. 

Here we prove the following theorem. 

Let $A$ be a finite abelian group. 

\begin{theorem}\label{maint}
For any $N$ and any irreducible 
representation $V$ of the quantum double 
$D(S_N\ltimes A^N)$, 
the higher Frobenius-Schur indicators $\nu_n(V)$ are integers.
\end{theorem}

In the case $A=1$, a proof of this theorem appears in
\cite{IMM}. 

{\bf Remark.} For basics on the quantum double, see \cite{Ka},
Section IX. 

\subsection{Proof of Theorem \ref{maint}}

Let us start with deriving an explicit formula for $\nu_n(V)$ 
(which is well known to experts, see e.g. \cite{KSZ}). 
Let $G$ be a finite group. For $D(G)$, one has 
$$
\Lambda=|G|^{-1}(\sum_{g\in G}g)\delta_1,
$$
where $\delta_g(x):=1$ if $g=x$ and $\delta_g(x)=0$ otherwise. 
So 
$$
\Lambda_1...\Lambda_n=|G|^{-1}\sum_g\sum_{x_1...x_n=1}g\delta_{x_1}g\delta_{x_2}...g\delta_{x_n}=
|G|^{-1}\sum_g\sum_{x_1...x_n=1}g^n\delta_{g^{1-n}x_1g^{n-1}}...\delta_{x_n}=
$$
$$
|G|^{-1}\sum_g\sum_{y: g^n=(y^{-1}g)^n}g^n\delta_y.
$$
Thus, 
$$
\nu_n(V)=|G|^{-1}{\rm Tr}_V(\sum_g\sum_{y: g^n=(y^{-1}g)^n}g^n\delta_y).
$$
Let us now recall the structure of irreducible representations 
of $D(G)$. Such a representation $V$ is attached to a conjugacy class 
$C$ in $G$ and an irreducible representation $W$ of the centralizer $Z_y$ of an element 
$y\in C$, and has the form $V=\oplus_{u\in C}V_u$, where $V_u=\delta_uV$, and $V_y=W$. 
So the formula for $\nu_n$ can be rewritten as 
$$
\nu_n(V)=|G|^{-1}\sum_{u\in C}\sum_{g\in G: g^n=(u^{-1}g)^n}{\rm Tr}_{V_u}(g^n).
$$
(note that $g^n$ is automatically in $Z_u$ as it commutes with $g$ and $u^{-1}g$). 
It is clear that the summands corresponding to all $u\in C$ are the same, so 
we get 
$$
\nu_n(V)=\frac{|C|}{|G|}\sum_{g\in G: g^n=(y^{-1}g)^n}{\rm Tr}_{W}(g^n)=
\frac{|C|}{|G|}\sum_{g,h\in G: g^n=h^n, gh^{-1}=y}{\rm Tr}_{W}(g^n).
$$

Let us multiply this formula by ${\rm Tr}_{W^*}(z)$ and sum over $W$. 
To keep track of the dependence on $W$, we will write $V(W)$ instead of $V$. 
By orthogonality of characters, we get 
$$
\sum_{W\in {\rm Irr}Z_y} 
\nu_n(V(W)){\rm Tr}_{W^*}(z)=\# \lbrace{(g,h)\in G^2: g^n=h^n=z, gh^{-1}=y\rbrace}.
$$
Thus, we see that integrality of $\nu_n(V)$ for all $V$ is equivalent to the statement that 
for any $y$ the function $f_y: Z_y\to \Bbb Z$ 
given by the formula 
$$
f_y(z)=\# \lbrace{(g,h)\in G^2: g^n=h^n=z, gh^{-1}=y\rbrace}
$$
is a virtual character of $Z_y$. This, in turn, is equivalent to saying that for any $s$ 
which is coprime to $|G|$, one has $f_y(z)=f_y(z^s)$. 

So the theorem follows from the following proposition. 

\begin{proposition}
Let $G=S_N\ltimes A^N$. Then for any $s$ 
which is coprime to $|G|$, one has $f_y(z)=f_y(z^s)$. 
\end{proposition}

The rest of the subsection is the proof of this proposition. 

First of all, $g,h,y$ commute with $z$. So if we write 
$z$ as $(z_1,...,z_q)$ where $z_j$ comprises all cycles of a
fixed type in $z$ (i.e. fixed length and monodromy
\footnote{If  $z=c\cdot (a_1,...,a_n)\in S_n\ltimes A^n$, 
where $c$ is a cycle of length $n$ and
  $a_i\in A$, then by the monodromy of $z$ we mean the element
  $a_1...a_n$ in $A$. It is invariant with respect to conjugation.}),
then we'll get $g=(g_1,...,g_q)$, $h=(h_1,...,h_q)$, $y=(y_1,...,y_q)$
accordingly (where $g_j,h_j,y_j$ are some elements commuting with $z_j$). 
Thus, $f_y(z)=f_{y_1}(z_1)...f_{y_q}(z_q)$. This shows that we may assume that $z$ has only cycles of 
some fixed length $m$ with fixed monodromy $u\in A$. 

So we can assume that $N=mr$ and $z=(c,c,...,c)\in (S_m\ltimes A^m)^r\subset G$, where $c=(12...m)(u,0,,,,0)$.
In this case, the centralizer $Z_z$ is $S_r\ltimes B^r$, where $B$ is the central extension 
of $\Bbb Z/m\Bbb Z$ by $A$ which is generated in $S_m\ltimes A^m$ by $c$ and the diagonal copy of $A$. 
So we have $g=\tau\cdot (a_1,...,a_r)$, 
$h=\theta\cdot (b_1,...,b_r)$, 
$y=\sigma\cdot (k_{\theta^{-1}(1)},...,k_{\theta^{-1}(r)})$, 
for $\tau, \theta\in S_r\subset S_{mr}$ (diagonal copy 
in $S_r^m\subset S_{mr}$), $\sigma=\tau \theta^{-1}$,
and $a_i,b_i,k_i\in B$.  
Then the equation $gh^{-1}=y$ says that 
\begin{equation}\label{eq1}
a_j-b_j=k_j,
\end{equation}
and the equations $g^n=h^n=z^s$ say that 
$\tau^n=\theta^n=1$ (so all the cycles of $\tau$ and $\theta$ have lengths dividing $n$), and 
for any cycle $K=(i_1,...,i_{d(K)})$ of $\tau$, one has 
\begin{equation}\label{eq2}
\frac{n}{d(K)}(a_{i_1}+...+a_{i_{d(K)}})=sc,
\end{equation}
while for any cycle $K=(j_1,...,j_{d(K)})$ of $\theta$, one has 
\begin{equation}\label{eq3}
\frac{n}{d(K)}(b_{j_1}+...+b_{j_{d(K)}})=sc.
\end{equation}

Let $\ell$ be the least common multiple of $n/d(K)$ for all cycles $K$. 
If equations (\ref{eq2}), (\ref{eq3}) have a solution, then 
$c=\ell \overline{c}$ for some (possibly non-unique) $\overline{c}\in B$.  

Thus, equations (\ref{eq2}) and (\ref{eq3}) can be rewritten as 
\begin{equation}\label{eq4}
a_{i_1}+...+a_{i_{d(K)}}=s\frac{\ell d(K)}{n}\overline{c}+v(K),\ \frac{n}{d(K)}v(K)=0
\end{equation}
for any cycle $K=(i_1,...,i_{d(K)})$ of $\tau$, and  
\begin{equation}\label{eq5}
b_{j_1}+...+b_{j_{d(K)}}=s\frac{\ell d(K)}{n}\overline{c}+v(K),\ \frac{n}{d(K)}v(K)=0
\end{equation}
for any cycle $K=(j_1,...,j_{d(K)})$ of $\theta$.

We can eliminate $a_j$ by setting $a_j=b_j+k_j$, so we get:
\begin{equation}\label{eq6}
b_{i_1}+...+b_{i_{d(K)}}=k_{i_1}+...+k_{i_{d(K)}}+s\frac{\ell d(K)}{n}\overline{c}+v(K),\ \frac{n}{d(K)}v(K)=0
\end{equation}
for any cycle $K=(i_1,...,i_{d(K)})$ of $\tau$, and  
\begin{equation}\label{eq7}
b_{j_1}+...+b_{j_{d(K)}}=s\frac{\ell d(K)}{n}\overline{c}+v(K),\ \frac{n}{d(K)}v(K)=0
\end{equation}
for any cycle $K=(j_1,...,j_{d(K)})$ of $\theta$.

For given $\tau,\theta$, equations (\ref{eq1}), (\ref{eq6}), (\ref{eq7}) 
are linear inhomogeneous equations in $b_j,v(K)$, 
whose associated homogeneous equations don't depend on $s$. 
So the number of solutions, if it is not zero, does not depend on $s$, and 
our job is just to show that the {\it solvability} of equations 
(\ref{eq1}), (\ref{eq6}), (\ref{eq7}) does not depend on $s$. 

Now we will need the following (well known) lemma. 

\begin{lemma} Let $B$ be a finite abelian group. 
Let $P=(P_i, 1\le i\le l_P)$ and $Q=(Q_j, 1\le j\le l_Q)$ be two set partitions of $\lbrace{1,...,r\rbrace}$. 
Let $p_i, 1\le i\le l_P$, $q_j, j=1,...,l_Q$ be elements of $B$. 
Then the system of equations  
$$
\sum_{k\in P_i}b_k=p_i, 1\le i\le l_P;\ 
\sum_{k\in Q_j}b_k=q_j, 1\le j\le l_Q;\ 
$$
has a solution $(b_1,...,b_r)\in B^r$
if and only if for every subset $S\subset \lbrace{1,...,r\rbrace}$ 
compatible with both $P$ and $Q$, one has 
$$
\sum_{i: P_i\subset S}p_i=\sum_{j: Q_j\subset S}q_j.
$$
\end{lemma}

\begin{proof}
Clearly, the condition is necessary, since both sides equal $\sum_{i\in S}b_i$. 
For sufficicency, it suffices to consider the case when $(P,Q)$ is indecomposable, 
i.e. any nonempty $S$ is the whole $\lbrace{1,...,r\rbrace}$. In this case, we just have one equation
$$
\sum_{i=1}^{l_P}p_i=\sum_{j=1}^{l_Q}q_j.
$$

Let $1_S$ be the characteristic function of a subset $S$. 
Our job is to show that 
if $a_i,b_j\in B^*$ are such that 
$$
\sum a_i1_{P_i}=\sum b_j1_{Q_j}
$$
then $a_i=b_j=\phi$ for some $\phi\in B^*$ and all $i,j$. 

To this end, 
let $f=\sum a_i1_{P_i}-\sum b_j1_{Q_j}$. Then $f(x)=a_{i(x)}-b_{j(x)}$, where $x\in P_{i(x)}$ and $x\in Q_{j(x)}$. 
So if $f=0$ then $a_{i(x)}=b_{j(x)}$. Consider the bipartite graph with vertices $i$ and $j$, where $i$ is connected to $j$ 
if $P_i\cap Q_j\ne \emptyset$. Since $(P,Q)$ is indecomposable, 
this graph is connected. So we see that if $f=0$ then $a_i=b_j=\phi$ for all $i$, 
$j$, as desired. 
\end{proof}

Now, the lemma easily implies the proposition. 
Indeed, let $P,Q$ be the partitions of $\lbrace{1,...,r\rbrace}$ into cycles 
of $\tau$ and $\theta$. Then the lemma implies that the condition of 
solvability of equations (\ref{eq1}), (\ref{eq6}), (\ref{eq7}) in $b_i$  
is a system of linear equations in the 
variables $v(K)$, which is independent on $s$, as desired 
(the terms which depend on $s$ cancel).

\section{Bounds on the number of irreducible representations of
the quantum double of a finite group}

\subsection{The result}

The goal of this section is to establish bounds for the number of
irreducible representations of the quantum double of a finite group,
answering a question asked to me by E. Rowell.

Let $f(n)$ be the maximal number of irreducible representations of
the quantum double $D(G)$ (over $\Bbb C$) of a finite group $G$ with $\le n$
conjugacy classes (or, equivalently, irreducible representations).
Note that $f(n)$ is well defined since by Landau's theorem (1903),
the set of such groups is finite.

\begin{theorem}
There are positive constants $C_1,C_2$ such that
$$
C_1n^{1/3}\log(n)\le \log f(n)\le C_2n\log^7(n).
$$
\end{theorem}

The proof, given in the next subsections, readily
follows from the results in the literature.

{\bf Remark.} We don't know if the upper bound holds for
semisimple Hopf algebras or fusion categories.

\subsection{The upper bound}

By a result of Pyber \cite{P} improved by Keller \cite{K},
there is $C>0$ such that $n\ge C\log |G|/(\log\log|G|)^7$.
for sufficiently large $n$.
On the other hand, $f(n)\le |G|^2$. This implies the statement.

\subsection{The lower bound}

Let $p$ be a prime. Consider the group $G$
of changes of variable
$x\to x+a_2x^2+....+a_{p+1}x^{p+1}$
modulo $x^{p+2}$ (this kind of groups is considered in \cite{KLG}).
This is a nilpotent group of order $p^p$, so
its exponential and logarithm
map are well defined, and ${\rm Lie}(G)$
is the Lie algebra
spanned over $\Bbb F_p$ by vector fields
$L_i:=x^{i+1}\partial_x$, $i=1,...,p$,
with commutator modulo $x^{p+2}$ (so we have $[L_i,L_j]=(j-i)L_{i+j}$,
where $L_k:=0$ if $k>p$).
Conjugacy classes correspond to adjoint orbits
in this Lie algebra, which are easily seen to be the orbit of $0$ and
the orbits of $aL_i+bL_{2i}$, where $a\ne 0$, and $i$
runs from $1$ to $p$. So the number of conjugacy classes
of $G$ is $\le p^3$.

On the other hand, consider the representations of $D(G)$.
They are parametrized by conjugacy classes of $G$ and irreducible
representations of centralizers $Z_g$. Take the conjugacy class of
$g=\exp(L_{(p+1)/2})$. Then ${\rm Lie}(Z_g)$ is spanned by
$L_{(p+1)/2}$,...,$L_p$, which commute, so $Z_g$ is
a vector space over $\Bbb F_p$ of dimension $(p+1)/2$.
This group has $p^{\frac{p+1}{2}}$ irreducible representations.
So we see that $\log f(p^3)\ge \frac{p+1}{2}\log p$.

Now for large $n$, find a prime $p$ such that $\frac{1}{2}n^{1/3}\le
p\le n^{1/3}$. Then since by definition $f$ is increasing,
we find that $\log f(n)\ge C_1n^{1/3}\log(n)$ for some constant $C_1$,
as desired.


\begin{thebibliography}{999}

\bibitem[IMM]{IMM} M. Iovanov, G. Mason, S. Montgomery, 
FSZ-groups and Frobenius-Schur Indicators of Quantum Doubles,
arXiv:1208.4153. 

\bibitem[Ka]{Ka} C. Kassel, Quantum groups, Gradiuate Texts in
  Mathematics, Springer, 1995. 

\bibitem[KSZ]{KSZ} Y. Kashina, Y. Sommerhäuser, Y. Zhu, On higher Frobenius–Schur indicators, 
Mem. Amer. Math. Soc. 181 (2005)

\bibitem[K]{K} T. M. Keller,
Finite groups have even more conjugacy classes,
 arXiv:0812.2590, Israel Journal of Mathematics
Volume 181, Number 1, 433-444.

\bibitem[KLG]{KLG} L. G. Kovacs and C. R. Leedham-Green,
Some normally monomial p-groups of maximal class
and large derived length, Q. J. Math. (1986) 37(1): 49-54.

\bibitem[P]{P} L. Pyber, Finite groups have many conjugacy classes,
Journal of the London Mathematical Society. Second Series 46 (1992), 239-249.

\end{thebibliography}
\end{document}